\documentclass[a4paper,12pt,reqno]{amsart}
\usepackage{amsfonts}
\usepackage{amsmath,amsthm,amssymb}

\nonstopmode \numberwithin{equation}{section}

\newtheorem{theorem}{Theorem}
\newtheorem{corollary}{Corollary}[section]

\allowdisplaybreaks

\begin{document}
\title[ Fractional Distributed Order Reaction-Diffusion \dots]{Solution of fractional Distributed Order Reaction-Diffusion Systems with Sumudu Transform.}
\author{K S Nisar}
\address{Department of Mathematics, College of Arts and Science-Wadi Aldawaser%
\\
Prince Sattam bin Abdulaziz University, Alkharj, Saudi Arabia}
\email{n.sooppy@psau.edu.sa, ksnisar1@gmail.com}
\author{Z.M. Gharsseldien}
\address{a) Department of Mathematics, College of Arts and Science-Wadi Aldawaser, 11991%
\\
Prince Sattam bin Abdulaziz University, Saudi Arabia}
\address{b) Department of Mathematics, Faculty of Science, 11884, Al-Azhar University, Cairo,Egypt.}
\email{gharsseldien\_z@yahoo.com }
\author{F.B.M. Belgacem}
\address{Department of Mathematics,  Faculty of Basic Education, PAAET, Al-Ardhiya, Kuwait}
\email{fbmbelgacem@gmail.com}
\keywords{Laplace Transform, Sumudu transform, Reaction-Diffusion Systems,
Mittag-Leffler function.}

\begin{abstract}
The solution of some fractional differential equations is the hottest topic in fractional calculus field. The fractional distributed order reaction-diffusion equation   is the aim of this paper. By applying integral transform to solve this type of fractional differential equations, we have obtained the analytical solution by using Laplace-Sumudu transform.

\textbf{Key words~~~:}Laplace Transform, Sumudu transform, Reaction-Diffusion Systems,
Mittag-Leffler function

\textbf{2010 Mathematics Subject Classification~~~:}44A10, 44A20, 35K57, 33E12 
\end{abstract}

\maketitle

\section{Introduction}

In recent decades, researchers care study and applications of fractional calculus, which is considered 
one of the most interesting topics in applied mathematics. The applications of fractional integral operator 
involving various special functions has found in various sub fields such as statistical distribution theory, 
control theory, fluid dynamics, stochastic dynamical system, plasma, image processing, nonlinear biological 
systems, astrophysics, and in quantum mechanics (see \cite{Balenu-b1}, \cite{Kiriyakova-b1}, \cite{Balenu-b2}). 
The great use of different transforms in quantum mechanics gives wide variety  of results, so it attracted 
mathematicians and physicists to pay more attention (\cite{Laskin1}-\cite{Laskin3}). 
Many more related works found in (\cite{Naber}, \cite{Saxena}, \cite{Saxena1},
\cite{Saxena2}, \cite{Saxena3},\cite{Haubold},\cite{Henry},\cite{Debnath1},] and there in.

\section{Mathematical background}

Starting this section, by providing the definition of Laplace transform and its inverse, then the definition of Sumudu transform  and its  inverse for the function of  function $f(x,t) $. these definitions are given as:

If a continuous or piecewise continuous function $f(x,t)$ is assumed as a function of  order ($e^t$) when $t\rightarrow \infty $, then the Laplace transform with respect to $t$  and its inverse with respect to $s$  is respectively given by
\begin{equation}
\mathcal{L}\{f(x,t);s\} =\mathcal{F}(x,s)=\int_{0}^{\infty }e^{-st}f(x,t) dt,~~~( t>0) ,~~~(x\in R)   \label{1}
\end{equation}
such that, $\Re(s)>0,$ and 
\begin{equation}
\mathcal{L}^{-1}\{\mathcal{F}(x,s)\} =f(x,t)=\frac{1}{2\pi i}\int_{\nu -i\infty }^{\nu +i\infty }e^{st}\mathcal{F}(x,s)ds,  \label{2}
\end{equation}
where $\nu \in \mathbb{R}$ is fixed.

In this paper, we concern our attention to two definitions based on Riemann-Liouville (R-L) definition of fractional integral of order $\nu >0$ which given by \cite{Miller},

\begin{eqnarray}
 \nonumber
 _{0}I_{t}^{\nu }f(x,t) &=& \frac{1}{\Gamma (\nu)}\int_{0}^{t }\frac{f(x,u)}{(t-u)^{1-\nu}} du ~~~~~ \text{where} \Re( \nu ) >0 \label{2a}  \\
 _{0}I_{t}^{0 }f(x,t) &=& f(x,t)
\end{eqnarray}
  \begin{itemize}
    \item ) The first one is Riemann-Liouville definition of fractional derivative of order $\alpha >0$, and it may be given as follows \cite{Samko}
\begin{equation}
_{0}D_{t}^{\alpha }f(x,t)=\frac{d^{m}}{dt^{m}}  \left ( _{0}I_{t}^{m-\alpha }f(x,t)\right ) ,~~~( t>0) ,~~~(m=[\alpha+1])   \label{2b}
\end{equation}
where $[\alpha]$ is the integer part of the number $\alpha $.
    \item ) The second is Caputo derivative \cite{Caputo}, given in the form
    \begin{equation}
_{0}^{c}D_{t}^{\alpha }f(x,t)=\left\{
                                \begin{array}{ll}
                                  _{0}I_{t}^{m-\alpha }f^{(m)}(x,t), & ~~~\text{where}~~ m-1<\alpha\leq m, Re(\alpha)>0,m\in N; \\
                                  f^{(m)}(x,t), & ~~~\text{where}~~ \alpha=m.
                                \end{array}
                              \right..\label{2c}
\end{equation}
where $ f^{(m)}(x,t)=\dfrac{\partial ^{m}f}{\partial t^{m}}$ is m$^{th}~$.
  \end{itemize}

 The Laplace transform of both definitions\cite{Oldham}:
\begin{eqnarray}
% \nonumber to remove numbering (before each equation)
  \mathcal{L}\{_{0}D_{t}^{\alpha }f(x,t);s\} &=& s^{\alpha}\mathcal{F}(x,s)-\sum_{r=1}^{m}s^{r-1}~_{0}D_{t}^{\alpha-r}f(x,t) \vert_{t=0}. \\
  \mathcal{L}\{ _{0}^{c}D_{t}^{\alpha }f( x,t) ;s\} &=& s^{\alpha}\mathcal{F}( x,s)-\sum_{r=0}^{m-1}s^{\alpha -r-1}f^{( r)}( x,0+)
\end{eqnarray}
where  $m-1<\alpha \leq m$, and $\Re(s)>0.$

\bigskip The Sumudu transform over the set functions

\begin{equation*}
A=\left\{ f( t) \left\vert \exists ~M,\tau _{1},\tau
_{2}>0,\left\vert f( t) \right\vert <Me^{\left\vert t\right\vert
/\tau _{j}}\right. ,~\text{if }t\in ( -1) ^{j}\times \lbrack
0,\infty )\right\}
\end{equation*}

is defined by

\begin{equation}
G( u) =S\left[ f( t) \right] =\int_{0}^{\infty
}f( ut) e^{-t}dt,~u\in ( -\tau _{1},\tau _{2})
\label{3}
\end{equation}

The detailed literature of Sumudu transform is found in (\cite{Asiru},\cite%
{Belgacem1},\cite{Belgacem2}). The following results due to \cite{VBL}

\begin{equation}
S^{-1}\left[ u^{\gamma -1}( 1-wu^{\beta }) ^{-\delta }\right]
=t^{\gamma -1}E_{\beta ,\gamma }^{\delta }( wt^{\beta })
\label{4}
\end{equation}

by interpreting it with the help of the formula $S^{-1}\left[ u^{n}:t\right] =%
\frac{t^{n}}{\Gamma ( n+1) }$ $,$gives%
\begin{equation}
S^{-1}\left[ \frac{u^{-\rho }}{u^{-\alpha }+\lambda u^{-\beta }+b}\right]
=\sum_{r=0}^{\infty }( -\lambda)^{r}t^{( \alpha -\beta )
r+\alpha -\rho }E_{\alpha ,\alpha +( \alpha -\beta ) r-\rho
+1}^{r+1}( -bt^{\alpha })  \label{5}
\end{equation}
and for $\lambda=0$ we get
\begin{equation}
S^{-1}\left[ \frac{u^{-\rho }}{u^{-\alpha }+b}\right]
=t^{\alpha -\rho }E^1_{\alpha ,\alpha -\rho
+1}( -bt^{\alpha })  \label{6}
\end{equation}
where $E_{\alpha ,\beta }^{\delta}( z) $ is the Mittag-Leffler function
\cite{Erdelyi} is in the form%
\begin{equation}
E_{\alpha ,\beta }^{\delta}( z) =\sum_{n=0}^{\infty }\frac{(\delta)_nz^{n}}{\Gamma
( n\alpha +\beta )n! },( \alpha ,\beta ,\delta \in \mathbb{C},\Re( \alpha
) >0,\Re( \beta ) >0,\Re( \delta) >0)  \label{Mittag}
\end{equation}
where $(\delta)_n$ is Pochhammer symbol.The Sumudu convolution theorem in \cite{Belgacem2} is
\begin{equation}
S^{-1}\left[ u \,S[\psi(t)]\, S[\phi(t)]\right]=(\psi\ast \phi)(t) = \int_0^t \psi(\xi)\,\phi(t-\xi)d\xi  \label{6a}
\end{equation}
The Sumudu transform of the Riemann-Liouville derivative \cite{Belgacem1} is
given as:
\begin{equation}
S\left[ _{0}D_{t}^{\alpha }f( x,t) ;u\right] =u^{-\alpha }\mathbb{F}( x,u) -\sum_{k=1}^{m}\frac{ \left [ _{0}D_{t}^{\alpha -k}f(x,t)\right ]_{t=0}}{u^{k}},~~( m-1<\alpha \leq m)  \label{6b}
\end{equation}
and for Caputo's derivative is given as:

\begin{equation}
S\left[ _{0}^{C}D_{t}^{\alpha }f( x,t) ;u\right] =u^{-\alpha }\mathbb{F}( x,u) -\sum_{k=0}^{m-1}\frac{1}{u^{\alpha -k}} \left [ \frac{\partial^k f(x,t)}{\partial t^k}\right ]_{t=0},~~( m-1<\alpha \leq m) \label{7}
\end{equation}

\section{Main Results}

Here we deals with the solution of fractional of reaction diffusion
equation

\begin{theorem}
The one dimensional fractional non-homogeneous reaction diffusion
system is defined by%
\begin{equation}
_{0}D_{t}^{\alpha }N( x,t) =\sum_{j=1}^{n}\mu _{j~x}D_{\theta
_{j}}^{\gamma _{j}}N( x,t) +\phi ( x,t)  \label{7a}
\end{equation}

where $t>0,\,x\in \mathbb{R},$ and, the parameters  $\alpha ,\,\theta _{1},...,\theta _{n},\,\gamma
_{1},...,\gamma _{n}$ are real. These parameters satisfy the following condition:%
\begin{equation}
\mu _{j}>0,\, 0<\gamma _{j}\leq 2,j=1,2,...,n,\left\vert \theta _{j}\right\vert
\leq \min_{1\leq j\leq n}( \gamma _{j},2-\gamma _{j}) ,1<\alpha
\leq 2  \label{8}
\end{equation}%
with initial conditions
\begin{equation}
\left(
  \begin{array}{c}
    _{0}D_{t}^{\alpha -1}N( x,t) \\
    _{0}D_{t}^{\alpha-2}N( x,t) \\
  \end{array}
\right)_{t=0}=\left(
                \begin{array}{c}
                  f( x) \\
                  g( x) \\
                \end{array}
              \right)
,\,x\in \mathbb{R},\,\lim_{x\rightarrow
\pm \infty }N( x,t) =0,t>0  \label{9}
\end{equation}
where $_{0}D_{t}^{\nu }N( x,t) $ is the R-L
fractional derivative of N$( x,t) $ with respect to t of order $%
\nu $. The Riesz-Feller (R-L) space fractional derivative $_{x}D_{\theta
_{1}}^{\gamma _{1}},...,_{x}D_{\theta _{n}}^{\gamma _{n}}$  with asymmetries $\theta
_{1},...,\theta _{n}~$ respectively,$\mu _{1,...,}\mu _{n}$ are arbitrary
constants, and, the functions $f( x) ,g( x) $ and $\phi ( x,t) $ are given. Then the solution of $( \ref{7a}) $ is:%
\begin{eqnarray}
N( x,t) &=&\frac{t^{\alpha -1}}{\sqrt{2\pi }}\int_{-\infty
}^{+\infty }e^{-ikx}E_{\alpha ,\alpha }[ -b(k)t^{\alpha }] f^{\ast
}( k) dk  \notag \\
&&+\frac{t^{\alpha-2 }}{\sqrt{2\pi }}\int_{-\infty }^{+\infty
}e^{-ikx}E_{\alpha ,\alpha -1}[ -b(k)t^{\alpha }] g^{\ast }(
k) dk  \notag \\
&&+\frac{1}{\sqrt{2\pi }}\int_{0}^{t}\xi ^{\alpha -1}\int_{-\infty }^{+\infty
}E_{\alpha ,\alpha }[ -b(k)\xi ^{\alpha }] e^{-ikx}\phi ^{\ast
}( k,t-\xi ) dkd\xi  \label{10}
\end{eqnarray}

where $E_{\alpha ,\beta }( .) $ is  defined in $( \ref{Mittag}) $
\end{theorem}

\begin{proof}
By using  Sumudu transform with the variable $t$ and applying
the given conditions $( \ref{8}) $ and $( \ref{9}) $ we get%

\begin{equation*}
S\left[ _{0}D_{t}^{\alpha }N( x,t) ;u\right] =S\left[
\sum_{j=1}^{n}\mu _{j~x}D_{\theta _{j}}^{\gamma _{j}}N( x,t) ;u%
\right] +S\left[ \phi ( x,t) ;u\right]
\end{equation*}%
\begin{equation*}
u^{-\alpha }\mathbb{N}( x,u) -\left[ \frac{_{0}D_{t}^{\alpha -1}N( x,0) }{u}\right] -\left[ \frac{_{0}D_{t}^{\alpha -2}N( x,0) }{u^{2}}\right]
=\sum_{j=1}^{n}\mu _{j~x}D_{\theta _{j}}^{\gamma _{j}}\mathbb{N}( x,u) +{\Phi}( x,u)
\end{equation*}%
\begin{equation*}
u^{-\alpha }\mathbb{N}( x,u) -u^{-1}f(x) -u^{-2}g(x)=\sum_{j=1}^{n}\mu _{j~x}D_{\theta _{j}}^{\gamma _{j}}\mathbb{N}( x,u) +{\Phi}( x,u)
\end{equation*}%
where $S[\phi(x,t);u]=\Phi(x,u)$ denotes Sumudu transform. The integral representation of (R-L) in $x$ domain  \cite{Saxena4} is
\begin{eqnarray*}
 % \nonumber to remove numbering (before each equation)
_{x}D_{\theta_j}^{\gamma_j}f(x) &=&  \frac{\Gamma (1+\gamma_j)}{\pi} \sin [(\gamma_j + \theta_j)\frac{\pi}{2}]\int_{0}^{\infty}\frac{f(x+\eta)-f(x)}{\eta^{1+\gamma_j}}d\eta  \\
&& +\frac{\Gamma (1+\gamma_j)}{\pi}\sin{[(\gamma_j - \theta_j)\frac{\pi}{2}]}\int_{0}^{\infty}\frac{f(x-\eta)-f(x)}{\eta^{1+\gamma_j}}d\eta
\end{eqnarray*}
By using Fourier-Transform with the variable x, we get%
\begin{eqnarray*}
F\left[ u^{-\alpha }\mathbb{N}( x,u)-u^{-1}f( x) -u^{-2}g( x) \right] &=&
F\left[\sum_{j=1}^{n}\mu _{j~x}D_{\theta _{j}}^{\gamma _{j}}\mathbb{N}( x,u) +\Phi( x,u) \right] \\
u^{-\alpha }\mathbb{N}^\ast( k,u) -u^{-1}f^{\ast}( k) -u^{-2}g^{\ast }( k) &=&-\sum_{j=1}^{n}\mu
_{j~}\Psi _{\theta _{j}}^{\gamma _{j}}( k) \mathbb{N}^{\ast}( k,u) +\Phi^{\ast}(k,u)
\end{eqnarray*}%
where * denoted by the Fourier transform. For R-L fractional derivative the Fourier transform is
\begin{equation}
F\left[_{x}D_{\theta _{j}}^{\gamma _{j}}f(x);k \right]=-\Psi _{\theta _{j}}^{\gamma _{j}}( k)f^{\ast}( k)
\label{aa}
\end{equation}
such that
\begin{equation}
\Psi _{\theta _{j}}^{\gamma _{j}}( k)=|k|^{\gamma_j}\exp[i ~\text
{sign}(k)\frac{\theta_j \pi}{2}],~ 0<\gamma_j\leq 2,~|\theta_j|\leq min\{\gamma_j,2-\gamma_j\}
\label{aa1}
\end{equation}
Solving $\mathbb{N}^{\ast}( k,u) $ yields,%
\begin{equation*}
\mathbb{N}^{\ast}( k,u) \left[ u^{-\alpha}+\sum_{j=1}^{n}\mu _{j~}\Psi _{\theta _{j}}^{\gamma _{j}}( k)\right] =u^{-1}f^{\ast }( k) +u^{-2}g^{\ast }( k)+
{\Phi^{\ast} }( k,u)
\end{equation*}%
\begin{equation*}
\mathbb{N}^{\ast}( k,u) =\frac{u^{-1}f^{\ast }(k) }{u^{-\alpha }+b(k) }+\frac{u^{-2}g^{\ast }(k) }{u^{-\alpha }+b(k)}+\frac{\Phi^{\ast} (k,u) }{u^{-\alpha }+b(k)}
\end{equation*}

where $b( k) =\sum_{j=1}^{n}\mu _{j~}\Psi _{\theta _{j}}^{\gamma_{j}}( k) $

Applying the the inverse Sumudu transform and using the convolution theorem for the last term $( \ref{5}) $
\begin{align*}
S^{-1}[\mathbb{N}^{\ast}( k,u)] &=f^{\ast }(k) S^{-1}\left[\frac{u^{-1} }{u^{-\alpha }+b(k) }\right]+g^{\ast }(k)\\
&\times S^{-1}\left[\frac{u^{-2} }{u^{-\alpha }+b(k)}\right]+S^{-1}\left[u \left(\frac{u^{-1}}{u^{-\alpha }+b(k)}\right)\Phi^{\ast} (k,u)\right]
\end{align*}
and the result given in $,$we get%
\begin{eqnarray*}
N^{\ast}( k,t) &=&t^{\alpha -1}E_{\alpha ,\alpha}[ -b(k)t^{\alpha }] f^{\ast }( k) +t^{\alpha-2}E_{\alpha ,\alpha-1 }[ -b(k)t^{\alpha }] g^{\ast }( k) \\
&&+\int_{0}^{t}\phi^{\ast}( k,t-\xi ) \xi ^{\alpha-1
}E_{\alpha ,\alpha}[ -b(k)\xi ^{\alpha }] d\xi
\end{eqnarray*}

Finally using the inverse Fourier transform, we get the required result.
\end{proof}

\subsection{Particular cases:}

If we take $\theta _{1}=\theta _{2}=...=\theta _{n}=0$ then, instead of $b(k)$ we get
\begin{equation*}
\sigma (k) =\sum_{j=1}^{n}\mu _{j}\Psi_{0}^{\gamma_{j}}(k)
= \sum_{j=1}^{n}\mu _{j}|k|^{\gamma_j}=C k^2
\end{equation*}
such that $C$ is constant given by $C=\sum_{j=1}^{n}\mu _{j}$ and we have the 
following corollaries.

\begin{corollary}
Consider%
\begin{equation}
_{0}D_{t}^{\alpha }N( x,t) =\sum_{j=1}^{n}\mu _{j~x}D_{0}^{\gamma
_{j}}N( x,t) +\phi ( x,t)  \label{11}
\end{equation}

which is the one dimensional non homogeneous unified fractional reaction diffusion model 
associated with time derivative $_{0}D_{t}^{\alpha }$ and Riez space derivatives $%
_{x}D_{0}^{\gamma _{1}},...,_{x}D_{0}^{\gamma n}$,  $n\in \mathbb{N}$ ,t$>0,x\in R;\alpha ,\gamma _{1},...,\gamma _{n}$ 
are real parameters with constraints%
\begin{equation}
\mu _{j}>0,0<\gamma _{j}\leq 2,j=1,2,...,n,1<\alpha \leq 2  \label{12}
\end{equation}

with initial condition
\begin{equation}
\left(
  \begin{array}{c}
    _{0}D_{t}^{\alpha -1}N( x,t) \\
    _{0}D_{t}^{\alpha-2}N( x,t) \\
  \end{array}
\right)_{t=0}=\left(
                \begin{array}{c}
                  f( x) \\
                  g( x) \\
                \end{array}
              \right)
,\,x\in \mathbb{R},\,\lim_{x\rightarrow
\pm \infty }N( x,t) =0,t>0  \label{13}
\end{equation}
then the following formula hold true
\begin{eqnarray}
N( x,t) &=&\frac{t^{\alpha -1}}{\sqrt{2\pi }}\int_{-\infty
}^{+\infty }e^{-ikx}E_{\alpha ,\alpha }[ -\sigma (k)t^{\alpha }] f^{\ast
}( k) dk  \notag \\
&&+\frac{t^{\alpha-2 }}{\sqrt{2\pi }}\int_{-\infty }^{+\infty
}e^{-ikx}E_{\alpha ,\alpha -1}[ -\sigma (k)t^{\alpha }] g^{\ast }(
k) dk  \notag \\
&&+\frac{1}{\sqrt{2\pi }}\int_{0}^{t}\xi ^{\alpha -1}\int_{-\infty }^{+\infty
}E_{\alpha ,\alpha }[ -\sigma (k)\xi ^{\alpha }] e^{-ikx}\phi ^{\ast
}( k,t-\xi ) dkd\xi  \label{14}
\end{eqnarray}
\end{corollary}

\begin{corollary}
If consider $g( x) =0$ under the condition
\begin{equation}
\mu _{j}>0,0<\gamma _{j}\leq 2,j=1,2,...,n,\left\vert \theta _{j}\right\vert
\leq \min_{1\leq j\leq n}( \gamma _{j},2-\gamma _{j}) ,0<\alpha
\leq 1  \label{15}
\end{equation}

with initial condition
\begin{equation}
_{0}D_{t}^{\alpha -1}N( x,0) =\delta ( x)
,_{0}D_{t}^{\alpha -2}N( x,0) =0,x\in \mathbb{R}%
,\lim_{x\rightarrow \pm \infty }N( x,t) =0,t>0  \label{16}
\end{equation}%
then the following formula hold true%
\begin{eqnarray}
N( x,t) &=&\frac{t^{\alpha -1}}{2\pi}\int_{-\infty
}^{+\infty }e^{-ikx}E_{\alpha ,\alpha }[ -b(k)t^{\alpha }]dk  \notag \\
&&+\frac{1}{\sqrt{2\pi }}\int_{0}^{t}\xi ^{\alpha -1}\int_{-\infty }^{+\infty
}E_{\alpha ,\alpha }[ -b(k)\xi ^{\alpha }] e^{-ikx}\phi ^{\ast
}( k,t-\xi ) dkd\xi  \label{17}
\end{eqnarray}
\end{corollary}

\begin{theorem}
\label{Th2} 
Consider the following fractional order unified one-dimensional
non-homogeneous reaction diffusion equation
\begin{equation}
_{0}D_{t}^{\alpha }N( x,t) +\lambda ~_{0}D_{t}^{\beta }N(x,t) =\sum_{j=1}^{n}\mu _{j~x}D_{\theta _{j}}^{\gamma _{j}}N(x,t) +\phi ( x,t)  \label{18}
\end{equation}
where $t>0,\lambda,x\in \mathbb{R},~\alpha ,\beta ,\theta _{1},...,\theta
_{n},\gamma _{1},...,\gamma _{n}$ are real parameters with the conditions:%
\begin{equation}
\mu _{j}>0,0<\gamma _{j}\leq 2,j=1,2,...,n,\left\vert \theta _{j}\right\vert
\leq \min_{1\leq j\leq n}( \gamma _{j},2-\gamma _{j}) ,1<\alpha
\leq 2,1<\beta \leq 2  \label{19}
\end{equation}%
with initial conditions%
\begin{eqnarray}
_{0}D_{t}^{\alpha -1}N( x,0)=f_{1}(x)
,\qquad _{0}D_{t}^{\alpha -2}N( x,0)= g_{1}(x),\notag \\
_{0}D_{t}^{\beta -1}N( x,0)=f_{2}(x)
,\qquad_{0}D_{t}^{\beta -2}N( x,0)=g_{2}(x)  \notag \\
 ,x\in \mathbb{R},\lim_{x\rightarrow \pm \infty}N( x,t) =0,t>0  \label{20}
\end{eqnarray}%
then the following formula holds true:
\begin{eqnarray*}
N(x,t) &=&\sum_{r=0}^{\infty} \left[ \frac{(-\lambda)^{r}t^{(\alpha-\beta)r+\alpha-1}}{\sqrt{2\pi}}\int_{-\infty}^{\infty}\left(e^{-ikx}\left[ f_1^{\ast}(k)+\lambda f_2^{\ast}(k)\right] E_{\alpha,\alpha+(\alpha -\beta ) r}^{r+1}[-b(k)t^{\alpha}]\right)dk \right]\\
&&+\sum_{r=0}^{\infty} \left[ \frac{(-\lambda)^{r}t^{(\alpha-\beta)r+\alpha-2}}{\sqrt{2\pi}}\int_{-\infty}^{\infty}\left(e^{-ikx}\left[ g_1^{\ast}( k)+\lambda g_2^{\ast}( k)\right] E_{\alpha,\alpha+(\alpha -\beta ) r-1}^{r+1}[-b(k)t^{\alpha}]\right)dk \right]\\
&&+\sum_{r=0}^{\infty }\left[ \frac{(-\lambda)^{r}}{\sqrt{2\pi}}\int_{-\infty}^{\infty}\left(e^{-ikx}\int_{0}^{t}\phi^{\ast}( k,t-\xi ) \xi^{(\alpha-\beta )r+\alpha -1 }E_{\alpha,\alpha+(\alpha-\beta)r}^{r+1}[-b(k)\xi^{\alpha }]d\xi\right )dk \right]
\end{eqnarray*}

\end{theorem}

\begin{proof}
Applying the Sumudu transform with respect to time variable $t$ and using
$( \ref{19}) $ and $( \ref{20}) $ we
get%
\begin{equation*}
S\left[ _{0}D_{t}^{\alpha }N( x,t) +a_{~0}D_{t}^{\beta }N(
x,t) ;u\right] =S\left[ \sum_{j=1}^{n}\mu _{j~x}D_{\theta
_{j}}^{\gamma _{j}}N( x,t) ;u\right] +S\left[ \phi (
x,t) ;u\right]
\end{equation*}%
\begin{eqnarray*}
&&u^{-\alpha }\mathbb{N}( x,u) -\left[ \frac{_{0}D_{t}^{\alpha -1}N( x,0) }{u}\right] -\left[ \frac{_{0}D_{t}^{\alpha -2}N( x,0) }{u^{2}}\right]\\
&&+\lambda u^{-\beta }\mathbb{N}( x,u)-\lambda\left[ \frac{_{0}D_{t}^{\beta -1}N( x,0) }{u}\right] -\lambda\left[ \frac{_{0}D_{t}^{\beta -2}N( x,0) }{u^{2}}\right]
=\sum_{j=1}^{n}\mu _{j~x}D_{\theta _{j}}^{\gamma _{j}}\mathbb{N}( x,u) +{\Phi}( x,u)
\end{eqnarray*}%
\begin{eqnarray*}
(u^{-\alpha }+\lambda u^{-\beta })\mathbb{N}( x,u) -\frac{(f_1(x)+\lambda f_2(x))}{u} -\frac{(g_1(x)+\lambda g_2(x))}{u^2}
=\sum_{j=1}^{n}\mu _{j~x}D_{\theta _{j}}^{\gamma _{j}}\mathbb{N}( x,u) +{\Phi}( x,u)
\end{eqnarray*}%
Applying the Fourier-Transform with respect to the space variable $x$ , we get%
\begin{eqnarray*}
&&F\left[(u^{-\alpha }+\lambda u^{-\beta })\mathbb{N}( x,u)-\frac{(f_1(x)+\lambda f_2(x))}{u} -\frac{(g_1(x)+\lambda g_2(x))}{u^2} \right]  \\
&&=F\left[\sum_{j=1}^{n}\mu _{j~x}D_{\theta _{j}}^{\gamma _{j}}\mathbb{N}( x,u) +\Phi( x,u) \right]  \\
&&\left[u^{-\alpha }+\lambda u^{-\beta }\right]\mathbb{N}^\ast( k,u)-\frac{f_1^{\ast}( k)+\lambda f_2^{\ast}( k)}{u} -\frac{g_1^{\ast}( k)+\lambda g_2^{\ast}( k)}{u^2}\\
&&=-\sum_{j=1}^{n}\mu_{j~}\Psi _{\theta _{j}}^{\gamma _{j}}( k) \mathbb{N}^{\ast}( k,u) +\Phi^{\ast}(k,u)
\end{eqnarray*}%

Solving $\mathbb{N}^{\ast}( k,u) $ yields,%
\begin{equation*}
\mathbb{N}^{\ast}( k,u) \left[ u^{-\alpha}+\lambda u^{-\beta }+\sum_{j=1}^{n}\mu _{j~}\Psi _{\theta _{j}}^{\gamma _{j}}( k)\right] =\frac{f_1^{\ast}( k)+\lambda f_2^{\ast}( k)}{u} +\frac{g_1^{\ast}( k)+\lambda g_2^{\ast}( k)}{u^2}+
{\Phi^{\ast} }( k,u)
\end{equation*}%
\begin{equation*}
\mathbb{N}^{\ast}( k,u) =\frac{u^{-1}(f_1^{\ast}( k)+\lambda f_2^{\ast}( k)) }{u^{-\alpha}+\lambda u^{-\beta }+b(k) }+\frac{u^{-2}(g_1^{\ast}( k)+\lambda g_2^{\ast}( k)) }{u^{-\alpha}+\lambda u^{-\beta }+b(k)}+\frac{\Phi^{\ast} (k,u) }{u^{-\alpha}+\lambda u^{-\beta }+b(k)}
\end{equation*}
where $b( k) =\sum_{j=1}^{n}\mu _{j~}\Psi _{\theta _{j}}^{\gamma
_{j}}( k) $
Applying the the inverse Sumudu transform and using the convolution theorem for the last term $( \ref{5}) $
\begin{eqnarray*}
S^{-1}[\mathbb{N}^{\ast}( k,u)] &=&\left(f_1^{\ast}( k)+\lambda f_2^{\ast}( k)\right) S^{-1}\left[\frac{u^{-1} }{u^{-\alpha}+\lambda u^{-\beta }+b(k) }\right]\\ &&+\left(g_1^{\ast}( k)+\lambda g_2^{\ast}( k)\right) S^{-1}\left[\frac{u^{-2} }{u^{-\alpha}+\lambda u^{-\beta }+b(k)}\right]\\
&&+S^{-1}\left[u \left(\frac{u^{-1}}{u^{-\alpha}+\lambda u^{-\beta }+b(k)}\right)\Phi^{\ast} (k,u)\right]
\end{eqnarray*}
we get%
\begin{eqnarray*}
N^{\ast}(k,t) &=&\left[ f_1^{\ast}(k)+\lambda f_2^{\ast}(k)\right] \sum_{r=0}^{\infty}(-\lambda)^{r}t^{(\alpha-\beta)r+\alpha -1 }E_{\alpha,\alpha+(\alpha -\beta ) r}^{r+1}[-b(k)t^{\alpha}]\\
&&+\left[ g_1^{\ast}( k)+\lambda g_2^{\ast}( k)\right] \sum_{r=0}^{\infty }(-\lambda)^{r}t^{(\alpha-\beta)r+\alpha-2}E_{\alpha,\alpha+(\alpha-\beta)r-1}^{r+1}[ -b(k)t^{\alpha }]\\
&&+\sum_{r=0}^{\infty }( -\lambda)^{r}\int_{0}^{t}\phi^{\ast}( k,t-\xi ) \xi^{(\alpha-\beta )r+\alpha -1 }E_{\alpha,\alpha+(\alpha-\beta)r}^{r+1}[-b(k)\xi^{\alpha }]d\xi
\end{eqnarray*}
Finally applying the inverse Fourier transform, we get
\begin{eqnarray*}
N(x,t) &=&\sum_{r=0}^{\infty} \left[ \frac{(-\lambda)^{r}t^{(\alpha-\beta)r+\alpha-1}}{\sqrt{2\pi}}\int_{-\infty}^{\infty}\left(e^{-ikx}\left[ f_1^{\ast}(k)+\lambda f_2^{\ast}(k)\right] E_{\alpha,\alpha+(\alpha -\beta ) r}^{r+1}[-b(k)t^{\alpha}]\right)dk \right]\\
&&+\sum_{r=0}^{\infty} \left[ \frac{(-\lambda)^{r}t^{(\alpha-\beta)r+\alpha-2}}{\sqrt{2\pi}}\int_{-\infty}^{\infty}\left(e^{-ikx}\left[ g_1^{\ast}( k)+\lambda g_2^{\ast}( k)\right] E_{\alpha,\alpha+(\alpha -\beta ) r-1}^{r+1}[-b(k)t^{\alpha}]\right)dk \right]\\
&&+\sum_{r=0}^{\infty }\left[ \frac{(-\lambda)^{r}}{\sqrt{2\pi}}\int_{-\infty}^{\infty}\left(e^{-ikx}\int_{0}^{t}\phi^{\ast}( k,t-\xi ) \xi^{(\alpha-\beta )r+\alpha -1 }E_{\alpha,\alpha+(\alpha-\beta)r}^{r+1}[-b(k)\xi^{\alpha }]d\xi\right )dk \right]
\end{eqnarray*}
\end{proof}
To get the results of theorem (1) we expand the summation into two parts $r=0$ and $\sum_{r=1}^{\infty}$, then we substitute by $\lambda=0$

\end{document}